\documentclass[10pt]{article}
 \font\smallit=cmti10

\usepackage{amssymb,latexsym,amsmath,epsfig,amsthm} 
\usepackage{enumitem,bigstrut,bigdelim,multirow,diagbox}

\makeatletter

\renewcommand{\@seccntformat}[1]{\csname the#1\endcsname. }

\makeatother

 \newtheorem{theorem}{Theorem}[section]
 \newtheorem{lemma}[theorem]{Lemma}
 
 \newtheorem{proposition}[theorem]{Proposition}

 \newtheorem{example}[theorem]{Example}

\begin{document}
\begin{center}
 {\bf On $p$-adic Transference Theorem}
 \vskip 30pt

 {\bf Chi Zhang}\\
 {\smallit Key Laboratory of Mathematics Mechanization, NCMIS, Academy of Mathematics and Systems Science, Chinese Academy of Sciences, Beijing 100190, People's Republic of China}\\
 {and}\\
 {\smallit School of Mathematical Sciences, University of Chinese Academy of Sciences, Beijing 100049, People's Republic of China}\\

 \vskip 10pt

 {\tt zhangchi171@mails.ucas.ac.cn}\\

 \end{center}
\vskip 30pt

\centerline{\bf Abstract} \par
Dual lattice is an important concept of Euclidean lattices. In 2024, Deng gave the definition to the concept of the dual lattice of a $p$-adic lattice from the duality theory of locally compact abelian groups. He also proved some important properties of the dual lattice of $p$-adic lattices, which can be viewed as $p$-adic analogues of the famous Minkowski's first, second theorems and transference theorems for Euclidean lattices. However, he only proved the lower bounds of the transference theorems and Minkowski's second theorem for $p$-adic lattices. The upper bounds are left as an open question. In this paper, we prove the upper bounds of the transference theorems and Minkowski's second theorem for $p$-adic lattices. We then prove that the dual basis of an orthogonal basis is also an orthogonal basis with respect to the maximum norm.

\vskip 10pt
2010 Mathematics Subject Classification: Primary 11F85.\par
Key words and phrases:  $p$-adic lattice, Local field, CVP, LVP, Transference theorem, Minkowski's theorem.

\noindent

\pagestyle{myheadings}

 \thispagestyle{empty}
 \baselineskip=12.875pt
 \vskip 20pt

\section{Transference Theorem}

In \cite{ref-2}, Deng gave the definition to the concept of the dual lattice of a $p$-adic lattice and then proved some important properties of the dual lattice of $p$-adic lattices, including the lower bound of the transference theorems for $p$-adic lattices. He left the upper bound as an open question.\par
The lower bound exists even if the $p$-adic lattice is not of full rank. However, the upper bound does not exist if the $p$-adic lattice is not of full rank. Consider the following example.\par
Let $k=\mathbb{Q}_p$ or $k=\mathbb{F}_p((T))$. Let $M$ be the maximum norm on $k^m$. Let the basis of $\Lambda$ be
$$B=\begin{pmatrix}
1 \\
\vdots \\
1 
\end{pmatrix}.$$
Then the dual basis
$$D=\frac{1}{m}\begin{pmatrix}
1 \\
\vdots \\
1 
\end{pmatrix}.$$
Hence
$$\tilde{\lambda}_{1}(\Lambda)\cdot\tilde{\lambda}_{1}(\Lambda^*)=\left|\frac{1}{m}\right|_p.$$
We can take $m=p^l$ for any large integer $l$.\par
Therefore, in the following subsections, we consider the $p$-adic lattice which is of full rank.

\subsection{Transference Theorem for the Maximum Norm}

Let $k=\mathbb{Q}_p$ or $k=\mathbb{F}_p((T))$. Let $M$ be the maximum norm on $k^m$, i.e., for $v=(v_1,\dots,v_m)\in k^m$, 
$$M(v)=\max_{1\le i\le m}{\left|v\right|_p}.$$
We consider the upper bound for the transference theorem respect to norm $M$ first. We need the following orthogonalization algorithm.\par\vskip 10pt

{\bf Algorithm} (orthogonalization with orthogonal bases).\par
{\bf Input:} an $N$-orthogonal basis $e_1,\dots,e_n$ of $k^n$, a lattice $\mathcal{L}=\mathcal{L}(\alpha_1,\dots,\alpha_m)$ in $k^n$.\par
{\bf Output:} an $N$-orthogonal basis of $\mathcal{L}$.
\begin{enumerate}
\item for $i=1$ to $m$ do:
\item\quad rearrange $\alpha_i,\dots,\alpha_m$ such that  $N(\alpha_i)=\max_{i\le k\le m}{N(\alpha_k)}$,
\item\quad rearrange $e_i,\dots,e_n$ such that  $N(a_{ii}e_i)=\max_{i\le j\le m}{N(a_{ji}e_j)}$,
\item\quad for $l=i+1$ to $m$ do:
\item\quad\quad $\alpha_l\leftarrow\alpha_l-\frac{a_{il}}{a_{ii}}\alpha_i.$
\end{enumerate}\par
Return $(\alpha_1,\dots,\alpha_m)$.\par\vskip 10pt

\begin{theorem}[\cite{ref-1}]\label{th-1.1}
The algorithm above outputs an $N$-orthogonal basis of $\mathcal{L}$.
\end{theorem}

Notice that the above theorem is proved only for $k=\mathbb{Q}_p$, but it is easy to see that the theorem also holds for $k=\mathbb{F}_p((T))$.\par
Now we can prove our main theorem of this subsection.

\begin{theorem}\label{th-1.2}
Let $k=\mathbb{Q}_p$ or $k=\mathbb{F}_p((T))$. Suppose $\Lambda$ is any $p$-adic lattice of full rank in $k^n$, $\Lambda^*$ is its dual. Then we have
$$\tilde{\lambda}_i(\Lambda)\cdot\tilde{\lambda}_{n+1-i}(\Lambda^*)=1$$
for $1\le i\le n$, where $\tilde{\lambda}_i(\Lambda)$ is the $i$-th successive maxima of the lattice $\Lambda$ and $\tilde{\lambda}_{n+1-i}(\Lambda)$ is the $(n+1-i)$-th successive maxima of the lattice $\Lambda^*$ with respect to the norm $M$.
\end{theorem}
\begin{proof}
Suppose $\Lambda=\mathcal{L}(B)$ and $D=(d_1,\dots,d_n)$ is the dual basis of $B=(b_1,\dots,b_n)$. We have $\Lambda^*=\mathcal{L}(D)$. Let the basis of $k^n$ be $e_1,\dots,e_n$. Then $e_1,\dots,e_n$ is an $M$-orthogonal basis of $k^n$ because for any vector $v=v_1e_1+\cdots+v_ne_n\in k^n$, the norm
$$M(v)=\max_{1\le i\le n}{\left|v_i\right|_p}=\max_{1\le i\le n}{\left|v_ie_i\right|_p}.$$
Therefore, we can run the algorithm ``orthogonalization with orthogonal bases'' to obtain the following $M$-orthogonal basis of $\Lambda$.
$$B=\begin{pmatrix}
b_{11} & 0 & \cdots & 0 \\
b_{21} & b_{22} & \cdots & 0 \\
\vdots & \vdots & \ddots & \vdots \\
b_{n1} & b_{n2} & \cdots & b_{nn}
\end{pmatrix}$$
where $\left|b_{ii}\right|_p=\max{\left\{\left|b_{jk}\right|_p:i\le k \le j\le n\right\}}$. Then we have $\tilde{\lambda}_i(\Lambda)=M(b_i)=\left|b_{ii}\right|_p$ for $1\le i\le n$.\par
Let
$$B^*=\begin{pmatrix}
B_{11} & 0 & \cdots & 0 \\
B_{12} & B_{22} & \cdots & 0 \\
\vdots & \vdots & \ddots & \vdots \\
B_{1n} & B_{2n} & \cdots & B_{nn}
\end{pmatrix}$$
be the adjoint matrix of $B$. We have
$$B_{ii}=\frac{\det{B}}{b_{ii}}$$
for $1\le i\le n$. and
$$D=(\frac{B^*}{\det{B}})^\mathsf{T}=\begin{pmatrix}
b_{11}^{-1} & * & \cdots & * \\
0 & b_{22}^{-1} & \cdots & * \\
\vdots & \vdots & \ddots & \vdots \\
0 & 0 & \cdots & b_{nn}^{-1}
\end{pmatrix}.$$
Write  $D=(d_{ij})$. Since
$$\left|B_{ji}\right|_p\le\left|\frac{\det{B}}{b_{ii}}\right|_p=\left|B_{ii}\right|_p,$$
we have $\left|d_{ij}\right|_p\le\left|d_{jj}\right|_p$ for all $1\le i\le j\le n$. Also, we have
$$\left|d_{nn}\right|_p=\left|b_{nn}^{-1}\right|_p\ge\left|d_{n-1,n-1}\right|_p=\left|b_{n-1,n-1}^{-1}\right|_p\ge\cdots\ge\left|d_{11}\right|_p=\left|b_{11}^{-1}\right|_p.$$
Hence $d_1,\dots,d_n$ is an orthogonal basis of $\Lambda^*$ according to the algorithm. Therefore, we have
$$\tilde{\lambda}_{n+1-i}(\Lambda^*)=M(d_i)=\left|b_{ii}^{-1}\right|_p$$
and
$$\tilde{\lambda}_i(\Lambda)\cdot\tilde{\lambda}_{n+1-i}(\Lambda^*)=M(b_i)\cdot M(d_i)=1$$
for $1\le i\le n$.
\end{proof}

\begin{example}
Let $k=\mathbb{Q}_2$ and $n=3$. Let the basis of $\Lambda$ be
$$B=\begin{pmatrix}
1 & 2 & 0 \\
0 & 2 & 2 \\
2 & 0 & 4
\end{pmatrix}.$$
We can run the algorithm ``orthogonalization with orthogonal bases'' to obtain an $M$-orthogonal basis of $\Lambda$,
$$B^{\prime}=\begin{pmatrix}
1 & 0 & 0 \\
0 & 2 & 0 \\
2 & -4 & 8
\end{pmatrix}.$$
Hence $\tilde{\lambda}_{1}(\Lambda)=1$, $\tilde{\lambda}_{2}(\Lambda)=\frac{1}{2}$ and $\tilde{\lambda}_{3}(\Lambda)=\frac{1}{8}$ with respect to the norm $M$. The dual basis
$$D=\frac{1}{8}\begin{pmatrix}
4 & -4 & 2 \\
2 & 2 & -1 \\
-2 & 2 & 1
\end{pmatrix}.$$
We can run the algorithm ``orthogonalization with orthogonal bases'' to obtain an $M$-orthogonal basis of $\Lambda^*$,
$$D^{\prime}=\frac{1}{8}\begin{pmatrix}
8 & -8 & 2 \\
0 & 4 & -1 \\
0 & 0 & 1
\end{pmatrix}.$$
Hence $\tilde{\lambda}_{1}(\Lambda^*)=8$, $\tilde{\lambda}_{2}(\Lambda^*)=2$ and $\tilde{\lambda}_{3}(\Lambda^*)=1$ with respect to the norm $M$. We have
$$\tilde{\lambda}_{1}(\Lambda)\cdot\tilde{\lambda}_{3}(\Lambda^*)=\tilde{\lambda}_{2}(\Lambda)\cdot\tilde{\lambda}_{2}(\Lambda^*)=\tilde{\lambda}_{3}(\Lambda)\cdot\tilde{\lambda}_{1}(\Lambda^*)=1.$$
\end{example}

\subsection{Transference Theorem for General Norms}

For a general norm $N$ on $k^m$, there is also an upper bound for the transference theorem. We need the following lemma and proposition.

\begin{lemma}[\cite{ref-3}]\label{le-1.4}
Let $N$ be a norm on $k^m$. Then the two norms $M$ and $N$ are equivalent, i.e., there are two positive constants $c_1$, $c_2$ such that
$$c_1N(v)\le M(v)\le c_2N(v)$$
for any $v\in k^n$.
\end{lemma}

\begin{proposition}[\cite{ref-1}]\label{pr-1.5}
Let $N$ be a norm on $k^m$. Let $\mathcal{L}=\mathcal{L}(\alpha_1,\dots,\alpha_n)$ be a lattice of rank $n$ in $k^m$ such that $N(\alpha_{1})\ge\cdots\ge N(\alpha_{n})$. Then $\tilde{\lambda}_1(\mathcal{L})=N(\alpha_1)$ and $\tilde{\lambda}_i(\mathcal{L})\le N(\alpha_i)$ for $2\le i\le n$.
\end{proposition}

Notice that the above lemma and proposition is proved only for $k=\mathbb{Q}_p$, but it is easy to see that the proposition also holds for $k=\mathbb{F}_p((T))$.\par
Now we can prove the transference theorem for a general norm $N$ on $k^m$.

\begin{theorem}\label{th-1.6}
Let $k=\mathbb{Q}_p$ or $k=\mathbb{F}_p((T))$. Suppose $\Lambda$ is any $p$-adic lattice of full rank in $k^n$, $\Lambda^*$ is its dual. Let $N$ be a norm on $k$. Then we have
$$\tilde{\lambda}_i(\Lambda)\cdot\tilde{\lambda}_{n+1-i}(\Lambda^*)\le c^{\prime}$$
for $1\le i\le n$, where $\tilde{\lambda}_i(\Lambda)$ is the $i$-th successive maxima of the lattice $\Lambda$, $\tilde{\lambda}_{n+1-i}(\Lambda)$ is the $(n+1-i)$-th successive maxima of the lattice $\Lambda^*$ with respect to the norm $N$ and $c^{\prime}$ is a positive constant depending on $N$.
\end{theorem}
\begin{proof}
Suppose $\Lambda=\mathcal{L}(B)$ and $D=(d_1,\dots,d_n)$ is the dual basis of $B=(b_1,\dots,b_n)$. We have $\Lambda^*=\mathcal{L}(D)$. Without loss of generality, we may assume 
$$B=\begin{pmatrix}
b_{11} & 0 & \cdots & 0 \\
b_{21} & b_{22} & \cdots & 0 \\
\vdots & \vdots & \ddots & \vdots \\
b_{n1} & b_{n2} & \cdots & b_{nn}
\end{pmatrix}$$
and $\left|b_{ii}\right|_p=\max{\left\{\left|b_{jk}\right|_p:i\le k \le j\le n\right\}}$. Notice that since $e_1,\dots,e_n$ may not be an $N$-orthogonal basis of $k^n$, the above basis is not an $N$-orthogonal basis of $\Lambda$ in general. However, we can still perform the transformation in the algorithm ``orthogonalization with orthogonal bases'' to obtain it.\par
Suppose that
$$N(b_{s_1})\ge N(b_{s_2})\ge\cdots\ge N(b_{s_n})$$
where $\{s_1,s_2,\dots,s_n\}=\{1,2,\dots,n\}$. Then by Proposition \ref{pr-1.5}, we have
$$\tilde{\lambda}_{i}(\Lambda)\le N(d_{s_i})$$
for $1\le i\le n$.\par
The dual basis $D$ is the same as in Theorem \ref{th-1.2}. For any $1\le i\le n$, let $1\le j\le i$ be the subscript such that
$$N(d_{s_j})=\max_{1\le k\le i}{N(d_{s_k})}.$$
By Proposition \ref{pr-1.5}, we have
$$\tilde{\lambda}_{n-i+1}(\Lambda^*)\le \max_{1\le k\le i}{N(d_{s_k})}=N(d_{s_j}).$$
By Lemma \ref{le-1.4}, two norms $M$ and $N$ are equivalent, i.e., there are two positive constants $c_1$ and $c_2$ such that
$$c_1N(v)\le M(v)\le c_2N(v)$$
for any $v\in k^n$. Then we have
\begin{equation*}
\begin{split}
\tilde{\lambda}_i(\Lambda)\cdot\tilde{\lambda}_{n-i+1}(\Lambda^*)&\le N(b_{s_i})\cdot N(d_{s_j})\\
&\le N(b_{s_j})\cdot N(d_{s_j})\\
&\le c_1^{-2}M(b_{s_j})\cdot M(d_{s_j})\\
&=c_1^{-2}
\end{split}
\end{equation*}
for $1\le i\le n$. We can take $c^{\prime}=c_1^{-2}$.
\end{proof}

{\bf Remark.} Since we can choose $c_1=\max_{1\le i\le m}{N(e_i)}$, we can choose
$$c^{\prime}=\left(\max_{1\le i\le m}{N(e_i)}\right)^{-2}.$$

\section{Minkowski's Second Theorem}

We can prove the following theorem, which can be viewed as a $p$-adic analogue of Minkowski's second theorem for Euclidean lattices.

\begin{theorem}Let $k=\mathbb{Q}_p$ or $k=\mathbb{F}_p((T))$. 
\begin{enumerate}
\item[(1)] Suppose $\Lambda$ is any $p$-adic lattice of full rank in $k^n$. Then we have
$$\prod_{i=1}^{n}{\tilde{\lambda}_{i}(\Lambda)}=\det(\Lambda),$$
where $\tilde{\lambda}_{i}(\Lambda)$ $(1\le i\le n)$ are the $i$-th successive maxima of the lattice $\Lambda$ with respect to the norm $M$.
\item[(2)] Suppose $\Lambda$ is any $p$-adic lattice of full rank in $k^n$.  Let $N$ be a norm on $k^n$. Then we have
$$\prod_{i=1}^{n}{\tilde{\lambda}_{i}(\Lambda)}\le c^{\prime\prime}\cdot\det(\Lambda),$$
where $\tilde{\lambda}_{i}(\Lambda)$ $(1\le i\le n)$ are the $i$-th successive maxima of the lattice $\Lambda$ with respect to the norm $N$.
\end{enumerate}
\end{theorem}
\begin{proof}
Suppose $\Lambda=\mathcal{L}(B)$. Same as Theorem \ref{th-1.2}, we can run the algorithm ``orthogonalization with orthogonal bases'' to obtain the following $M$-orthogonal basis of $\Lambda$.
$$B=(b_1,b_2,\dots,b_n)=\begin{pmatrix}
b_{11} & 0 & \cdots & 0 \\
b_{21} & b_{22} & \cdots & 0 \\
\vdots & \vdots & \ddots & \vdots \\
b_{n1} & b_{n2} & \cdots & b_{nn}
\end{pmatrix}$$
where $\left|b_{ii}\right|_p=\max{\left\{\left|b_{jk}\right|_p:i\le k \le j\le n\right\}}$. Then we have $\tilde{\lambda}_i(\Lambda)=M(b_i)=\left|b_{ii}\right|_p$ for $1\le i\le n$. Therefore,
$$\prod_{i=1}^{n}{\tilde{\lambda}_{i}(\Lambda)}=\prod_{i=1}^{n}{M(b_i)}=\prod_{i=1}^{n}{\left|b_{ii}\right|_p}=\det(\Lambda).$$
This proves (1). Below, we prove (2).\par
By Proposition \ref{pr-1.5} and Lemma \ref{le-1.4}, we have
$$\prod_{i=1}^{n}{\tilde{\lambda}_{i}(\Lambda)}\le\prod_{i=1}^{n}{N(b_i)}\le c_1^{-n}\cdot\prod_{i=1}^{n}{M(b_i)}=c_1^{-n}\cdot\det(\Lambda).$$
We can take $c^{\prime\prime}=c_1^{-n}$.
\end{proof}

{\bf Remark.} Since we can choose $c_1=\max_{1\le i\le m}{N(e_i)}$, we can choose
$$c^{\prime\prime}=\left(\max_{1\le i\le m}{N(e_i)}\right)^{-n}.$$

\begin{example}
Let $K=\mathbb{Q}_2(\pi)$, where $\pi$ is a root of $f(x)=x^3-2$. Let $N$ be the $2$-adic absolute value on $K$. Since $f$ is an Eisenstein polynomial, $K$ is totally ramified of degree $3$ and
$$N(\pi)=2^{-\frac{1}{3}}.$$
Let $e_1=\pi$, $e_2=\pi^2$ and $e_3=2$. Then we can choose
$$c^{\prime}=\left(\max_{1\le i\le 3}{N(e_i)}\right)^{-2}=2^{\frac{2}{3}}.$$
and
$$c^{\prime\prime}=\left(\max_{1\le i\le 3}{N(e_i)}\right)^{-3}=2.$$
Let the basis of $\Lambda$ be
$$B=\begin{pmatrix}
1 & 2 & 0 \\
0 & 2 & 2 \\
2 & 0 & 4
\end{pmatrix}.$$
Since $\pi,\pi^2,2$ is an $N$-orthogonal basis of $K$, we can run the algorithm ``orthogonalization with orthogonal bases'' to obtain an $N$-orthogonal basis of $\Lambda$,
$$B^{\prime}=\begin{pmatrix}
1 & 0 & 0 \\
0 & 2 & 0 \\
2 & -4 & 8
\end{pmatrix}.$$
Hence $\tilde{\lambda}_{1}(\Lambda)=2^{-\frac{1}{3}}$, $\tilde{\lambda}_{2}(\Lambda)=2^{-\frac{5}{3}}$ and $\tilde{\lambda}_{3}(\Lambda)=\frac{1}{16}$ with respect to the norm $N$. The dual basis
$$D=\frac{1}{8}\begin{pmatrix}
4 & -4 & 2 \\
2 & 2 & -1 \\
-2 & 2 & 1
\end{pmatrix}.$$
We can run the algorithm ``orthogonalization with orthogonal bases'' to obtain an $N$-orthogonal basis of $\Lambda^*$,
$$D^{\prime}=\frac{1}{8}\begin{pmatrix}
8 & 0 & 2 \\
0 & 0 & -1 \\
0 & 4 & 1
\end{pmatrix}.$$
Hence $\tilde{\lambda}_{1}(\Lambda^*)=4$, $\tilde{\lambda}_{2}(\Lambda^*)=1$ and $\tilde{\lambda}_{3}(\Lambda^*)=2^{-\frac{1}{3}}$ with respect to the norm $N$. We have
$$\tilde{\lambda}_{1}(\Lambda)\cdot\tilde{\lambda}_{3}(\Lambda^*)=2^{-\frac{2}{3}}\le 2^{\frac{2}{3}},$$
$$\tilde{\lambda}_{2}(\Lambda)\cdot\tilde{\lambda}_{2}(\Lambda^*)=2^{-\frac{5}{3}}\le 2^{\frac{2}{3}},$$
and
$$\tilde{\lambda}_{3}(\Lambda)\cdot\tilde{\lambda}_{1}(\Lambda^*)=\frac{1}{4}\le 2^{\frac{2}{3}}.$$
On the other hand, we have
$$\tilde{\lambda}_{1}(\Lambda)\cdot\tilde{\lambda}_{2}(\Lambda)\cdot\tilde{\lambda}_{3}(\Lambda)=2^{-6}\le2^{-3}= 2\cdot\det(B),$$
and
$$\tilde{\lambda}_{1}(\Lambda^*)\cdot\tilde{\lambda}_{2}(\Lambda^*)\cdot\tilde{\lambda}_{3}(\Lambda^*)=2^{\frac{5}{3}}\le2^{5}=2\cdot\det(D).$$
\end{example}

\section{Dual of Orthogonal Basis}

It is interesting to ask whether the dual basis of an orthogonal basis is still an orthogonal basis. It turns out that for the norm $M$, the answer is affirmative. However, this conjecture does not hold for all norms.\par
We need the following theorems.

\begin{theorem}[\cite{ref-1}]\label{th-2.1}
Let $N$ be a norm on $k^m$. Let $\mathcal{L}=\mathcal{L}(\alpha_1,\dots,\alpha_n)$ be a lattice of rank $n$ in $k^m$ with an $N$-orthogonal basis $\alpha_1,\dots,\alpha_n$. If $\beta_1,\dots,\beta_n$ can be obtained from $\alpha_1,\dots,\alpha_n$ by the following operations:
\begin{enumerate}
\item $\alpha_i\leftarrow k\alpha_i$ for some $k\in\mathbb{Z}_p\setminus p\mathbb{Z}_p$,
\item $\alpha_i\leftrightarrow\alpha_j$,
\item $\alpha_i\leftarrow\alpha_i+k\alpha_j$ for some $k\in\mathbb{Z}_p$ such that $N(k\alpha_j)\le N(\alpha_i)$,
\end{enumerate}
then $\beta_1,\dots,\beta_n$ is also an $N$-orthogonal basis of $\mathcal{L}$.
\end{theorem}

\begin{theorem}[\cite{ref-1}]\label{th-2.2}
Let $N$ be a norm on $k^m$. Let $\mathcal{L}$ be a lattice of rank $n$ in $k^m$. Then any two  $N$-orthogonal bases of $\mathcal{L}$ can be obtained from each other by the following operations:
\begin{enumerate}
\item $\alpha_i\leftarrow k\alpha_i$ for some $k\in\mathbb{Z}_p\setminus p\mathbb{Z}_p$,
\item $\alpha_i\leftrightarrow\alpha_j$,
\item $\alpha_i\leftarrow\alpha_i+k\alpha_j$ for some $k\in\mathbb{Z}_p$ such that $N(k\alpha_j)\le N(\alpha_i)$.
\end{enumerate}
\end{theorem}

Notice that the above theorems are proved only for $k=\mathbb{Q}_p$, but it is easy to see that these theorems also hold for $k=\mathbb{F}_p((T))$.\par
Now we can prove our main theorem of this section.

\begin{theorem}\label{th-2.3}
Suppose $\Lambda$ is any $p$-adic lattice of full rank in $k^n$, $\Lambda^*$ is its dual. Suppose $\Lambda=\mathcal{L}(B)$ and $D$ is the dual basis of $B$. If $B$ is an $M$-orthogonal basis of $\Lambda$, then $D$ is an $M$-orthogonal basis of $\Lambda^*$.
\end{theorem}
\begin{proof}
According to the proof of Theorem \ref{th-1.2}, if the $M$-orthogonal basis of $\Lambda$ is of the form
$$B^{\prime}=\begin{pmatrix}
b_{11}^{\prime} & 0 & \cdots & 0 \\
b_{21}^{\prime} & b_{22}^{\prime} & \cdots & 0 \\
\vdots & \vdots & \ddots & \vdots \\
b_{n1}^{\prime} & b_{n2}^{\prime} & \cdots & b_{nn}^{\prime}
\end{pmatrix}$$
where $\left|b_{ii}^{\prime}\right|_p=\max{\left\{\left|b_{jk}^{\prime}\right|_p:i\le k \le j\le n\right\}}$, then the dual basis of $B$ will be of the form
$$D^{\prime}=\begin{pmatrix}
d_{11}^{\prime} & d_{12}^{\prime} & \cdots & d_{1n}^{\prime} \\
0 & d_{22}^{\prime} & \cdots & d_{2n}^{\prime} \\
\vdots & \vdots & \ddots & \vdots \\
0 & 0 & \cdots & d_{nn}^{\prime}
\end{pmatrix}.$$
where $\left|d_{ii}^{\prime}\right|_p=\max{\left\{\left|d_{jk}^{\prime}\right|_p:1\le j\le k \le i\right\}}$. Hence it is an $M$-orthogonal basis of $\Lambda^*$.\par
Now let $B$ be a general $M$-orthogonal basis of $\Lambda$. By Theorem \ref{th-2.2}, $B$ can be obtained from $B^{\prime}$ by the following operations:
\begin{enumerate}
\item $b_i\leftarrow kb_i$ for some $k\in\mathbb{Z}_p\setminus p\mathbb{Z}_p$,
\item $b_i\leftrightarrow b_j$,
\item $b_i\leftarrow b_i+kb_j$ for some $k\in\mathbb{Z}_p$ such that $N(kb_j)\le N(b_i)$.
\end{enumerate}
Write $B=B^{\prime}E$ where $E=E_1E_2\cdots E_s$ and $E_i$ $(1\le i\le s)$ is a fundamental column transformation matrix of one of the above three forms. Then
$$D=D^{\prime}(E^{-1})^{\mathsf{T}}=D^{\prime}(E_1^{-1})^{\mathsf{T}}(E_2^{-1})^{\mathsf{T}}\cdots (E_s^{-1})^{\mathsf{T}}.$$
Let
$$B_t=B^{\prime}E_1E_2\cdots E_t=\left(b_1^{(t)},\dots,b_n^{(t)}\right)$$
and
$$D_t=D^{\prime}(E_1^{-1})^{\mathsf{T}}(E_2^{-1})^{\mathsf{T}}\cdots (E_t^{-1})^{\mathsf{T}}=\left(d_1^{(t)},\dots,d_n^{(t)}\right)$$
be the basis and dual basis after the $t$-th fundamental column transformation.\par
Set $B_0=B^{\prime}$ and $D_0=D^{\prime}$. We claim that $d_1^{(t)},\dots,d_n^{(t)}$ is an $M$-orthogonal basis of $\Lambda^*$ and $M(b_i^{(t)})\cdot M(d_i^{(t)})=1$ for any $1\le i\le n$ and $0\le t\le s$. When $t=0$ the claim is proved in Theorem \ref{th-1.2}. Suppose that the claim holds for some $0\le t\le s-1$.\par
If $E_{t+1}$ is the operation $1$, then $(E_{t+1}^{-1})^{\mathsf{T}}$ will be the operation $d_i^{(t)}\leftarrow k^{-1}d_i^{(t)}$, also of the form 1. By Theorem \ref{th-2.1}, $d_1^{(t+1)},\dots,d_n^{(t+1)}$ is an $M$-orthogonal basis of $\Lambda^*$. On the other hand, we have
$$M\left(b_i^{(t+1)}\right)\cdot M\left(d_i^{(t+1)}\right)=M\left(b_i^{(t)}\right)\cdot M\left(d_i^{(t)}\right)=1.$$
Hence the claim holds for $t+1$.\par
If $E_{t+1}$ is the operation $2$, then $(E_{t+1}^{-1})^{\mathsf{T}}$ will be the operation $d_i^{(t)}\leftrightarrow d_j^{(t)}$. also of the form 2. By Theorem \ref{th-2.1}, $d_1^{(t+1)},\dots,d_n^{(t+1)}$ is an $M$-orthogonal basis of $\Lambda^*$. On the other hand, we have
$$M\left(b_i^{(t+1)}\right)\cdot M\left(d_i^{(t+1)}\right)=M\left(b_j^{(t)}\right)\cdot M\left(d_j^{(t)}\right)=1$$
and
$$M\left(b_j^{(t+1)}\right)\cdot M\left(d_j^{(t+1)}\right)=M\left(b_i^{(t)}\right)\cdot M\left(d_i^{(t)}\right)=1.$$
Hence the claim holds for $t+1$.\par
If $E_{t+1}$ is the operation $3$, then
$$N\left(kb_j^{(t)}\right)\le N\left(b_i^{(t)}\right)$$
and $(E_{t+1}^{-1})^{\mathsf{T}}$ will be the operation $d_j^{(t)}\leftarrow d_j^{(t)}-kd_i^{(t)}$. Since
$$N\left(-kd_i^{(t)}\right)=\left|k\right|_p\cdot N\left(b_i^{(t)}\right)^{-1}\le N\left(b_j^{(t)}\right)^{-1}=N\left(d_j^{(t)}\right),$$
this operation is also of the form 3. By Theorem \ref{th-2.1}, $d_1^{(t+1)},\dots,d_n^{(t+1)}$ is an $M$-orthogonal basis of $\Lambda^*$. On the other hand, since $b_1^{(t)},\dots,b_n^{(t)}$ and $d_1^{(t)},\dots,d_n^{(t)}$ are $M$-orthogonal bases of $\Lambda$ and $\Lambda^*$ respectively, we have
\begin{equation*}
\begin{split}
M\left(b_i^{(t+1)}\right)\cdot M\left(d_i^{(t+1)}\right)&=M\left(b_i^{(t)}+kb_j^{(t)}\right)\cdot M\left(d_i^{(t)}\right)\\
&=M\left(b_j^{(t)}\right)\cdot M\left(d_j^{(t)}\right)\\
&=1.
\end{split}
\end{equation*}
and
\begin{equation*}
\begin{split}
M\left(b_j^{(t+1)}\right)\cdot M\left(d_j^{(t+1)}\right)&=M\left(b_j^{(t)}\right)\cdot M\left(d_j^{(t)}-kd_i^{(t)}\right)\\
&=M\left(b_j^{(t)}\right)\cdot M\left(d_j^{(t)}\right)\\
&=1.
\end{split}
\end{equation*}
Hence the claim holds for $t+1$.\par
Therefore, by induction, $D=D_s$ is an $M$-orthogonal basis of $\Lambda^*$.
\end{proof}

For other norms, here is a simple counterexample:
\begin{example}
Let $K=\mathbb{Q}_2(\pi)$, where $\pi$ is a root of $f(x)=x^2-2$. Let $N$ be the $2$-adic absolute value on $K$. Since $f$ is an Eisenstein polynomial, $K$ is totally ramified of degree $2$ and
$$\left|\pi\right|_2=2^{-\frac{1}{2}}.$$
Let the basis of $K$ be $e_1=1$ and $e_2=\pi$. Let $\Lambda=\mathcal{L}(B)$ where
$$B=(b_1,b_2)=\begin{pmatrix}
1 & 0  \\
1 & 1  \\
\end{pmatrix}.$$
Then we have $b_1=1+\pi$ and $b_2=\pi$. According to the algorithm ``orthogonalization with orthogonal bases'', $b_1,b_2$ is an $N$-orthogonal basis of $\Lambda$.\par
The dual basis of $B$ is
$$D=(d_1,d_2)=\begin{pmatrix}
1 & -1  \\
0 & 1  \\
\end{pmatrix}.$$
Then we have $d_1=1$ and $d_2=\pi-1$. However, since
$$\left|d_1+d_2\right|_2=2^{-\frac{1}{2}}<1=\max{\left\{\left|d_1\right|_2,\left|d_2\right|_2\right\}},$$
$d_1,d_2$ is not an $N$-orthogonal basis of $\Lambda^*$.
\end{example}

\end{document}